\documentclass[12pt]{article} 
\usepackage[utf8]{inputenc}
\usepackage[T2A]{fontenc}\usepackage[english]{babel}
\usepackage{amsthm, amsmath, amssymb, amsbsy, amscd, amsfonts, latexsym, euscript,epsf,stackrel}
\usepackage{authblk}
\usepackage{graphicx} 

\usepackage{epsfig}
\newtheorem{theorem}{Theorem}[section]
\newtheorem{corollary}[theorem]{Corollary}

\newtheorem{proposition}[theorem]{Proposition}

\newtheorem{question}[theorem]{Question}

\newtheorem{definition}[theorem]{Definition}
\newtheorem{example}[theorem]{Example}
\newtheorem{remark}[theorem]{Remark}
\newtheorem*{ack}{Acknowledgement}
\usepackage[all,cmtip]{xy}
\usepackage{xcolor}
\usepackage{xfrac}
\usepackage{pb-diagram}

\newcommand{\Alex}{\operatorname{Alex}}
\newcommand{\Aut}{\operatorname{Aut}}

\newcommand{\Conj}{\operatorname{Conj}}
\newcommand{\Core}{\operatorname{Core}}

\newcommand{\Inn}{\operatorname{Inn}}

\newcommand{\T}{\operatorname{T}}

\newcommand{\R}{\operatorname{R}}

\newcommand{\id}{\mathrm{id}}

\def\bea{\begin{eqnarray}}
\def\eea{\end{eqnarray}}
\def\nn{\nonumber}

\begin{document}
\title{$n$-valued quandles and associated bialgebras
}

\author[1,2,3]{Valeriy G. Bardakov\footnote{bardakov@math.nsc.ru}}
\author[3]{Tatyana A. Kozlovskaya\footnote{t.kozlovskaya@math.tsu.ru}}
\author[4,5]{Dmitry~V.~Talalaev\footnote{dtalalaev@yandex.ru}}
\affil[1]{\small Sobolev Institute of Mathematics, Novosibirsk 630090, Russia.
}
\affil[2]{\small Novosibirsk State Agrarian University, Dobrolyubova street, 160, Novosibirsk, 630039, Russia.
}
\affil[3]{\small Regional Scientific and Educational Mathematical Center of Tomsk State University,
36 Lenin Ave., 14, 634050, Tomsk, Russia.  
}
\affil[4]{\small Lomonosov Moscow State University, 
119991, Moscow, Russia.
}
\affil[5]{\small Demidov Yaroslavl State University, 150003, Sovetskaya Str. 14
}


\maketitle

\begin{abstract}
 The principal aim of this article is to introduce and study $n$-valued quandles and  $n$-corack bialgebras.
We elaborate the basic methods of this theory, reproduce the coset construction known in the theory of $n$-valued groups. We also consider a construction  of  $n$-valued quandles using $n$-multi-quandles. In contrast to the case of $n$-valued groups  this construction turns out to be quite rich in algebraic and  topological applications.  An important part of the work is the study of the properties of $n$-corack bialgebras those role is analogous to the group bialgebra.


\end{abstract}

\newpage
\tableofcontents

\section{Introduction}

The notion of an $n$-valued group was introduces in 1971 by V.~M.~Buchstaber and S.~P.~Novikov \cite{BN}  in the theory of characteristic classes of vector bundles.  The product of a pair of elements in such an algebraic system is an $n$-multi-set, the set of $n$ points with multiplicities. This generalization of a group appeared to be fruitful in topology \cite{BN}, algebraic geometry \cite{BR1}, representation theory, combinatorics, dynamical systems and number theory \cite{BV}. An appropriate survey on $n$-valued groups and its applications  can be found in \cite{B}. 

The $n$-ary operations are quite common in different areas. The Yang-Baxter map could be considered as a 2-valued operation - an operation of biarity (2,2) in PROP theory terminology \cite{MCL}. In this particular case we have a pair of binary operations on a set.
In the theory of algebraic systems it is common that several operations have equal arity. We give some examples of this situation. 
A skew brace  is a set with two group operations, which satisfy appropriate axioms (\cite{Rump}-\cite{GV}). In  \cite{BNY} it were introduced brace systems as a set with a family of group operations with some cosistency conditions.
 J.-L.~Loday \cite{Loday2} introduced a dimonoid as a set with two semigroup operations, which are related by a set of axioms. In \cite{Loday2} it was presented  a construction of a free dimonoid generated by a given set.
Dimonoids are examples of duplexes, which were introduced by T.~Pirashvili in \cite{Pir}. A duplex is an algebraic system with two associative binary operations (these operations may not be related to each other). 
T.~Pirashvili  constructed a free duplex generated by a given set.

In the present article we are studying  (semi-)group  systems  $\mathcal{G} = (G, *_i, i \in I)$, that are algebraic systems such that  $(G, *_i)$ is a  (semi-)group for any $i \in I$. An example of a semi-group system with two operations is a duplex. We will call a  (semi-)group  systems $\mathcal{G}$ by $I$-multi-semigroup (respectively, $I$-multi-group) if the operations satisfy the following axioms
$$
(a *_i b) *_j c = a *_i (b *_j c),~~ a, b, c \in G,~~i, j \in I. 
$$
If $|I| = n$, then we will call this structure an $n$-multi-semigroup (resepctivle, $n$-multi-group).
N.~A.~Koreshkov \cite{Kor} calls an $n$-multi-semigroup as an $n$-tuple semigroup.

In the present article we  investigate connections between $n$-multi-groups  and $n$-valued groups.
If we have a group system $\mathcal{G} = (G, *_i, i \in I)$, where $|I| = n$, we can  define a $n$-valued multiplication 
$$
a * b = [a *_1 b, a *_2 b, \ldots,  a *_n b],~~a, b \in G,
$$ 
and study conditions under which the  algebraic system $(G, *)$ is a $n$-valued group.
We prove  that if in all groups 
$(G, *_i)$, $i = 1, 2, \ldots, n$ all units are equal  and $(G, *)$ is an $n$-valued group, then $*_i = *_j$ for all $1 \leq i, j \leq n$. It means, in particular, that in some sense multi-groups and multi-valued groups are different algebraic systems.

The second part of our article is devoted to quandles. Recall that a quandle is an algebraic system with a binary operation that satisfies axioms encoding the three Reidemeister moves of planar diagrams of links in the 3-space. 
Quandles were introduced by   Joyce \cite{Joyce} and Matveev \cite{Matveev}, who showed that link quandles are complete invariants 
of non-split links up to orientation of the ambient 3-space. Also, these objects are userful 
 in other  areas of mathematics: in  group theory, set-theoretic solutions to the Yang-Baxter equations and Hopf algebras \cite{Andr}, discrete integrable systems.
 Automorphisms of quandles, which reveal a lot about their internal structures, have been investigated in much detail in a series of papers \cite{BDS, BarTimSin}.

In the present article we introduce and study $n$-valued racks and quandles.  We suggest some constructions for them. In particular, a coset quandle, which is similar to
the coset group in the category of groups. Second construction is new,  we prove that if a set $X$ is equipped with $n$ quandle operations $*_i$ such that
\begin{equation} \label{ass}
(x *_i y) *_j z= (x *_j z) *_i (y *_j z),~~~x, y, z \in X,
\end{equation}
then an $n$-valued multiplication 
$$
x * y = [x *_1 y, x *_2 y, \ldots, x *_n y]
$$
defines  an $n$-valued quandle structure on $X$. We will call a set $X$ with a set of quandle operations $(X, *_1, *_2, \ldots, *_n)$ satisfying axioms of mixed associativity (\ref{ass}) a $n$-multi-quandle. Such axioms arise in \cite{BF} in studying multiplications of quandle structures. 
In this paper one considered quandle systems $\mathcal{Q} = (Q, *_i, i \in I)$, where  $(Q, *_i)$ is a quandle for any $i \in I$, and defined a multiplication $*_i *_j$ by the rule
$$
p (*_i *_j) q = (p *_i q) *_j q,~~p, q \in Q.
$$
Generically  $(Q, *_i *_j)$ is not a quandle, but if one imposes additinally
\begin{equation} \label{quand}
(x *_i y) *_j z= (x *_j z) *_i (y *_j z),~~(x *_j y) *_i z= (x *_i z) *_j (y *_i z),~~~x, y, z \in Q,
\end{equation}
then  $(Q, *_i *_j)$ and  $(Q, *_j *_i)$ are quandles. The term $n$-multi-quandle introduced V.~G.~Turaev \cite{T}  and gave them a  topological interpretation.

In the last part of the article 
we develop the theory of bialgebras associated to an $n$-valued quandle. The consistency condition for the multiplcation and comultiplication in such bialgebras is different from those in Hopf algebras. 
The self-distributivity axiom of a rack $X$ reads as 
\bea
(c * b) * a = (c * a) * (b * a),~~a, b, c \in X.\nn
\eea
As was shown in \cite{ABRW} the natural way to extend the self-distributivity to $\Bbbk[X]$ is the following relation
\bea
(c * b) * a = \sum_{(a)} (c * a^{(1)}) * (b * a^{(2)}),\nn
\eea
where we use the Sweedler notation for comultiplication:
\bea
\Delta a=\sum_i a_i^{(1)}\otimes a_i^{(2)}=\sum a^{(1)}\otimes a^{(2)}.\nn
\eea
The main result of this section is that the space of functions on an $n$-quandle is $n$-corack bialgebra. This statement is an analog of an $n$-Hopf algebra structure related to an $n$-valued group \cite{BR1}.

The article is organized sa follows.  In Section \ref{sec-prelim} we recall the definition of an $n$-valued group, recall two constructions of these groups, introduce  $n$-multi-groups, gives  example, formulate some properties of these objects.  In particular, in Proposition \ref{equel} is shown that if in a group system   $(G, *_1, *_2, \ldots, *_n)$  the unit elements of the operations 
 $*_i$ and  $*_j$ coincide,  then the operations $*_i$ and $*_j$ are equal. Example \ref{BE}  shows connection between  solution  of the Braid equation with  2-multi-semi-groups. On a set of squire  matrices over a field we define some deformations of the matrix product and get a set of operations. Proposition \ref{propety} give conditions under which we can construct a  2-multi-group.

In Section \ref{quand} we recall a definitions of racks and quandles, introduce $n$-valued racks and $n$-valued quandles, give some example.
We define the coset construction of an $n$-valued quandle, which is  an analogue of the group coset construction.

In Section \ref{ordqun}
we give definitions of  $n$-multi-rack and quandle, provide some constructions. In Theorem \ref{coset} we show that any $n$-multi-rack ($n$-multi-quandle) defines a $n$-valued rack (respectively, $n$-valued quandle). It contrast with the group case, in which if we have a $n$-multi-group such that the group operations have the same unit, then this multi-group is a $n$-valued group if and only if all group operations are equal. We construct linear 2-multi-racks on the set of integers in Proposition \ref{mr} and prove in Proposition \ref{BE}  that some of these 2-multi-racks give solution to the Braid equation.

In Section \ref{RBA} we discuss a possibility to define a comultiplication on rack rings and define rack bialgebra.

\bigskip


\section{$n$-valued groups and  $n$-multi-groups} \label{sec-prelim}

At first,  we recall a definition and construction of $n$-valued groups, which can be found in \cite{B}. 
\\
Let $X$ be a non-empty set. An $n$-{\it valued multiplication} on $X$ is a map
$$
\mu \colon X \times X \to (X)^n = Sym^nX,~~\mu(x, y) = x * y = [z_1, z_2, \ldots, z_n],~~z_k = (x*y)_k.
$$
The next axioms are natural generalization of group axioms.

{\it Associativity}. The $n^2$-sets:
$$
[x * (y * z)_1, x * (y * z)_2,  \ldots, x * (y * z)_n], ~~[ (x * y)_1 * z, (x * y)_2 * z,  \ldots, (x * y)_n * z]
$$
coincide for all $x, y, z \in X$.

{\it Unit}. There exists an element $e \in X$ such that 
$$
e * x = x * e = [x, x, \ldots, x]
$$
 for all $x \in X$.

{\it Inverse}. It is defined a map $inv \colon X \to X$ such that 
\bea
e \in inv(x) * x \qquad \mbox{and} \qquad e \in x * inv(x)  \nn
\eea
for all $x \in X.$
\begin{definition}
\label{n-valued}
An $n$-{\it valued group} is a quadruple $\mathcal{X} = (X, \mu, e, inv)$ such that the map $\mu$ is associative, has a unit and an inverse. 
\end{definition}

Let us recall two constructions of  $n$-valued groups.
 Let $G$ be a group with multiplication $m$, $A \leq \Aut(G)$ and $|A| = n$. The set of orbits  $X = G / A$  can be equipped with an $n$-multiplication $\mu \colon X \times X \to (X)^n$ defined by the rule
$$
\mu(x, y) = \pi (m (\pi^{-1}(x), \pi^{-1}(y))),
$$
where $\pi \colon G \to X$ is the canonical projection. Then $X$ with multiplication $\mu$ is  a $n$-valued group with the unit  $e_X = \pi(e_G)$ and the inverse $inv(x)$ of $x \in X$ is $\pi (( \pi^{-1}(x))^{-1})$. This $n$-valued group is called the {\it coset group} of the pair $(G, A)$.

The second construction of an $n$-valued group is  called a {\it double coset group}  of a pair $(G, H)$,
 where $G$ is a group and $H$ - its subgroup of cardinality  $n$. Denote by $X$  the space of double coset classes $H \backslash G / H$.
One could define the  $n$-valued multiplication $\mu \colon X \times X \to (X)^n$
by the formula
$$
\mu(x, y) = \{H g_1 H\} * \{H g_2 H\} = [\{H g_1 h g_2 H,~~h \in H].
$$
 and the inverse element $inv_X(x) = \{H g^{-1} H\}$, where 
$x =\{H g H\}$.

\medskip

An algebraic system  $\mathcal{G} = (G, *_i, i \in I)$, where $(G, *_i)$ is a (semi-)group  for any $i \in I$ is said to be a 
(semi-)group system. 
A semigroup system with two operations is called a duplex. We call $\mathcal{G}$ by $I$-multi-(semi-)group if the operations are related by the axiom of mixed associativity,
$$
(a *_i b) *_j c = a *_i (b *_j c),~~ a, b, c \in S,~~i, j \in I. 
$$
A multi-semigroup with $n$ operations is called by $n$-tuple semigroup in \cite{Kor}.

\begin{remark}
In \cite{T} it were defined several homogeneous algebraic systems which are generalization of  multi-(semi-)group.
\end{remark}

The next proposition shows that the definition of an $n$-multi-group with equal unit elements is trivial.

\begin{proposition} \label{equel}
If the unit elements of an $n$-multi group $(G, *_1, *_2, \ldots, *_n)$ coincide
 for some $i, j \in \{ 1, 2, \ldots, n \}$:   $e_i = e_j,$ then the corresponding operations $*_i = *_j$ also coincide.
\end{proposition}

\begin{proof}
Let $e = e_i = e_j$, then if we put $y = e$ in the axiom of mixed associativity, 
$$
(x *_i y) *_j z = x *_i (y *_j z),
$$
we get $x  *_j z = x *_i  z$ for any $x, z \in G$. Hence, $*_i = *_j$.
\end{proof}


Let us give some examples of  $n$-multi groups.

\begin{example} (Pencil of products) 
Let $A$ be a vector space with two binary associative algebraic operations $\circ_1$,  $\circ_2$, such that the next two axioms hold
\bea\label{cons}
(a \circ_1 b) \circ_2 c = a \circ_1 (b \circ_2 c),~~~(a \circ_2 b) \circ_1 c = a \circ_2 (b \circ_1 c),
\eea
for all $a, b, c \in A$. In particular, if $\circ_1 = \circ_2$, then these axioms fulfill. 

One can define a $2$-valued multiplication  on $A$:
$$
a * b = [a \circ_1 b, a \circ_2 b],~~a, b \in A.
$$
This operation is associative.
\end{example}

The mixed associativity axioms  (\ref{cons}), which represent the  consistency condition, are equivalent to Hochschild cocycle condition and  allows to define deformations of both operations $\circ_1$ and $\circ_2$ by the formulas
$$
a (\circ_1)_t b = a \circ_1 b + t (a \circ_2 b),~~~a (\circ_2)_t b = t( a \circ_1 b) + a \circ_2 b.
$$

Further we will use this idea for construction some operations on the set of matrices. 
Let $M_n(\Bbbk)$ be a matrix ring over a field $\Bbbk$. Let us fix a matrix $M \in M_n(\Bbbk)$ and define a set of multiplications $m_t \colon M_n(\Bbbk) \times M_n(\Bbbk) \to M_n(\Bbbk)$ 
$$
m_t (A, B) = A B + t A M B,~~t \in \Bbbk,~~A, B \in M_n(\Bbbk).
$$ 
Then we can define a 2-valued multiplication
$$
\mu(A, B) = (m_{t_1}(A, B), m_{t_2}(A, B)),  ~~~m_{t_i}(A, B) =  A B + t_i A M B,~~t_i \in \Bbbk,~~i=1,2.
$$

We have an obvious
\begin{proposition} \label{propety}
1) The 2-valued multiplication  $\mu(A, B)$ is associative.

2) The product $A B$ in $M_n(\Bbbk)$ has a unit $E$ given by the unit matrix. The product $A M B$ has the unit element $M^{-1}$, which is the inverse matrix in the ring $M_n(k)$ 

3) The inverse to $A$ under the multiplication $\nu(A, B) = A M B$ is equal to $M^{-1} A^{-1} M^{-1}$.
\end{proposition}



\begin{example} (Solution to the Braid equation) \label{BE}
Let $(S, \cdot)$ be a semigroup with a unit. Then $R(a, b) = (1, ab)$ gives a degenerate solution to the Braid equation:
$$
(R \times \id) (\id \times R) (R \times \id) = (\id \times R) (R \times \id)  (\id \times R).
$$
 If we introduce  a new operation on $S$ by the rule 
$a \circ b = 1$, then $(S, \circ)$ is a semigroup without unit. We can define a 2-valued multiplication  $a * b =  [a \circ b, a b]$, $a, b \in S$. This multiplication is not associative. Indeed, if $c \not= 1$, then 
$$
(a \circ b) c  \not= a \circ (b c).
$$
\end{example}

\bigskip



\section{$n$-valued racks and quandles} \label{quand}

\subsection{Definition and examples}

Let us recall the definitions of rack and quandle structures (see \cite{Joyce, Matveev}).

\begin{definition}
A {\it quandle} is a non-empty set $Q$ with a binary operation $(x,y) \mapsto x * y$ satisfying the following axioms:
\begin{enumerate}
\item[(Q1)]{\it Idempotency}:  $x*x=x$ for all $x \in Q$.
\item[(Q2)]  {\it Invertibility}: for any $x,y \in Q$ there exists a unique $z \in Q$ such that $x=z*y$.
\item[(Q3)] {\it Self-distributivity}: $(x*y)*z=(x*z) * (y*z)$ for all $x,y,z \in Q$.
\end{enumerate}
An algebraic system satisfying only (Q2) and (Q3)  is called a {\it rack}. 
\end{definition}
Many interesting examples of quandles come from group theory, demonstrating deep connection between these subjects.

\par

\begin{itemize}
\item If $G$ is a group, $m$ is an integer,  then the binary operation $a*_m b= b^{-m} \, a \,  b^m$ turns $G$ into the quandle $\Conj_m(G)$ called the $m$-{\it conjugation quandle} on $G$. 
If $m=1$, this quandle is called {\it conjugation quandle} and is denoted $\Conj(G)$.
\item A group $G$ with the binary operation $a*b= b a^{-1} b$ turns the set $G$ into the quandle $\Core(G)$ called the {\it core quandle} of $G$. In particular, if $G= \mathbb{Z}_n$, the cyclic group of order $n$, then it is called the {\it dihedral quandle} and denoted by $\R_n$.
\item Let $G$ be a group and $\phi \in \Aut(G)$. Then the set $G$ with binary operation $a * b = \phi(ab^{-1})b$ forms a quandle $\Alex(G,\phi)$ referred as the  {\it generalized Alexander quandle} of $G$ with respect to $\phi$.
\end{itemize}
\medskip

A quandle  $Q$ is called {\it trivial} if $x*y=x$ for all $x, y \in Q$.  Unlike groups, a trivial quandle can have arbitrary number of elements. We denote the $n$-element trivial quandle by $\T_n$ and an arbitrary trivial quandle by $\T$.

Notice that the axioms (Q2) and (Q3) are equivalent to the map $S_x \colon Q \to Q$ given by $S_x(y)=y*x$ being an automorphism of $Q$ for each $x \in Q$. These automorphisms are called {\it inner automorphisms}, and the group generated by all such automorphisms is denoted by $\Inn(X)$.


By analogy with $n$-valued groups let us introduce $n$-valued racks and $n$-valued quandles.

\begin{definition} \label{many-rack}

{\it  An $n$-valued rack} is a triple $\mathcal{X} = (X, *, \bar{*})$ in which $X$ is a non-empty set with $n$-valued multiplications $*$ and $\bar{*}$ such that
$$
x * y = [(x * y)_1, (x * y)_2, \ldots, (x * y)_n],~~x \bar{*} y = [(x \bar{*} y)_1, (x \bar{*} y)_2, \ldots, (x \bar{*} y)_n],~~x, y \in X,
$$
i.e. $x * y$ and $x \bar{*} y$ are multisets in $Sym^n(X)$. The next axioms hold

(M1)  {\it Invertibility}:  for any $x, y \in X$ the element $x$ lies in the $n^2$-multiset $(x * y) \bar{*} y$ and in the $n^2$-multiset
$(x \bar{*} y) * y$.

(M2) {\it Self-distributivity}: for any $x, y, z \in X$ the  $n^2$-multiset $(x * y) * z$ is a subset of the $n^3$-multiset $(x * z) * (y * z)$.

An $n$-valued rack  $\mathcal{X} = (X, *, \bar{*})$ is said to be an {\it $n$-valued quandle} if the next axiom holds

(M3)  {\it Idempotency}:  for any $x \in X$ the multiset $x * x$ contains $x$.
\end{definition}

\begin{example}
Let $G$ be a group and $I$ be a subset of integer numbers. For any $i \in I$ define $i$-conjugation quandle $Conj_i(G)$ on $G$. If $I$ is finite and contains $n$ elements, we can define $n$-valued multiplication
$$
g * h = [g *_i h,~~i \in I],~~g, h \in G.
$$
In this case $g \bar{*} h = [g \bar{*}_i h,~~i \in I]$ and it is easy to check that it is an $n$-valued quandle. 
\end{example}

\subsection{Coset construction} \label{coset}
There is an analogue of  the group coset construction  for quandles.

 Let $Q$ be a quandle with multiplication $m$ and its inverse $\bar{m}$ that means that for any $q, h \in Q$ the following holds
$$
\bar{m}(m(q, h), h) =  m(\bar{m}(q, h), h) = q.
$$
Further,  let $A \subset \Aut(Q)$ be a subgroup of order  $n$.
Then the set of orbits  $X = Q/A$ can be equipped with an $n$-valued multiplication $\mu \colon X \times X \to (X)^n$ by the rule
$$
\mu(x, y) = x * y = \pi (m (\pi^{-1}(x), \pi^{-1}(y))),
$$
where $\pi \colon Q \to X$ is the canonical projection. 

As in the case of groups (see \cite{B}) it is easy to prove the next theorem-definition.

\begin{theorem} 
 The multiplication $\mu$ defines an $n$-valued quandle  structure on the orbit space
$X = Q / A,$ called the  {\it coset $n$-valued quandle} of $(G, A)$, with the inverse operation
$$
 x \bar{*} y = \pi (\bar{m} (\pi^{-1}(x), \pi^{-1}(y))).
$$ 
\end{theorem}

\begin{example}\label{coset-Q}
Let us consider the conjugation quandle $Q = Conj(S_3)$ on the symmetric group $S_3$ and a subgroup of inner automorphisms $A \subset Inn(Q)$ that is generated by multiplication on $s_1$. It is evident that $A \cong S_2$. The orbits of $Q$ under the action of $A$  are
\bea
\{ 1 \}, \{s_1\}, \{s_2,s_1 s_2 s_1\}, \{s_1 s_2, s_2 s_1\}.\nn
\eea

Let us denote these classes by $x_0$, $x_1$, $x_2$, $x_3$, respectively. The multiplication table for this $2$-valued quandle is given by
$$
\begin{tabular}{|c||c|c|c|c|}
    \hline
$Q/A$ & $x_0$ & $x_1$ & $x_2$ &$ x_3 $\\
  \hline \hline
$x_0$ & $[x_0,x_0]$ & $[x_0,x_0]$ & $[x_0,x_0]$ & $[x_0,x_0]$ \\
$x_1 $& $[x_1,x_1]$ & $[x_1,x_1]$ & $[x_2,x_2]$ & $[x_2,x_2]$ \\
$x_2$ & $[x_2,x_2]$ & $[x_2,x_2]$ & $[x_1,x_2]$ & $[x_1,x_2]$\\
$x_3 $& $[x_3,x_3]$ & $[x_3,x_3]$ & $[x_3,x_3]$ & $[x_3,x_3]$\\
  \hline
\end{tabular}.
$$

\end{example}

\begin{example}
Recall a  construction of coset 2-valued group on $\mathbb{Z}$ (see \cite{B}). 
 Let $(\mathbb{Z}, +)$ be the infinite cyclic group, $A = \{\id, -\id\}$ be its group of automorphisms. The set of orbits 
$$
X = \{ \{0\}, \{\pm a \}~|~a \in \mathbb{N} \}
$$
can be identified with the set of non-negative integers $\mathbb{Z}_+$ with the $2$-valued multiplication 
$$
\mu \colon \mathbb{Z}_+ \times \mathbb{Z}_+ \to (\mathbb{Z}_+)^2
$$
given by the formula
$$
x * y = [x+y, |x-y|],~~x, y \in \mathbb{Z}_+.
$$. a

Construct coset 2-valued quandle on $Q = Core(\mathbb{Z})$. 
Let $\varphi \in \Aut(Q)$, $\varphi(a) =-a$, $a \in \mathbb{Z}$. The set  of orbits: $X = \{ \{a, -a \}~|~a \in \mathbb{Z} \}$ with   a 2-valued multiplication
$$
\{a, -a \} * \{b, -b \} = [\{2b-a, -2b+a \}, \{2b+a, -2b-a \}],~~a, b, c \in \mathbb{Z},
$$
gives a coset 2-valued quandle. 
We can  identified the set of orbits with the set of non-negative integers $\mathbb{Z}_+$ with the $2$-valued multiplication 
$$
* \colon \mathbb{Z}_+ \times \mathbb{Z}_+ \to (\mathbb{Z}_+)^2
$$
given by the formula
$$
a * b = [2b + a, |2 b - a|],~~a, b \in \mathbb{Z}_+.
$$
\end{example}

\bigskip


\section{Multi-racks} \label{ordqun}

Recall the definition of an $I$-multi-rack (see \cite{T}). A non-empty set $X$ with  rack (quandle) operations $*_i$, $i \in I$ is called an $I$-multi-rack (respectively, an $I$-multi-quandle) if for any $i, j \in I$ the distributivity axioms hold
$$
(x *_i y) *_j z= (x *_j z) *_i (y *_j z),~~~x, y, z \in X. 
$$
If $|I| = n$, then we will call it an $n$-multi-rack (respectively, an $n$-multi-quandle). If $(X, *_1, *_2, \ldots, *_n)$ is an $n$-multi-rack, then we can define an $n$-valued rack $(X, *)$ with the operation 
$$
x * y = [x *_1 y, x*_2 y, \ldots, x *_n y],~~x, y \in X.
$$

As was shown in Proposition \ref{equel}, if two units
 $e_i$  and  $e_j$  in  an $n$-multi-group $(G, *_1, *_2, \ldots, *_n)$ coincide, then the  operations $*_i$ and $*_j$ coincide too.
The following theorem shows that in the case of quandles the situation is different.

\begin{theorem} 
\label{coset}
Let $X$ be a non-empty set with $n$ quandle (rack) operations $*_i$ such that for any pair $1 \leq i, j \leq n$ of indeces the following identities hold
\begin{equation} \label{distr}
(x *_i y) *_j z= (x *_j z) *_i (y *_j z),~~~x, y, z \in X. 
\end{equation}
Then  an $n$-valued multiplication 
$$
x * y = [x *_1 y, x *_2 y, \ldots, x *_n y]
$$
defines  an $n$-valued quandle (correspondingly, rack) structure on $X$.
\end{theorem}

\begin{proof}
Follows from the definition of an $n$-valued quandle.
\end{proof}

In particular, if we take the powers of a quandle operation, then they satisfies the axioms (\ref{distr}) and we get 

\begin{corollary}
 Let $(Q, \circ)$ be an $n$-quandle (it means that for any $q, h \in Q$ we have  $q *^n h = q$). Let us define  an $n$-valued multiplication on $Q$ by the formula
$$
x * y = [x \circ y, x \circ^2 y, \ldots, x \circ^n y],~~~x, y \in Q.
$$
Then
$(Q, *)$ is an $n$-valued quandle.
\end{corollary}

\begin{remark}
The principal stimulus for us for introducing $n$-valued quandles is the notion of  $n$-valued group. But there is another appearance of multivalued structure in the context of quandle product which was introduced in \cite{BF}.

Suppose that  we have two quandle structures $(X, *_1)$ and  $(X, *_2)$ on the set $X$. Then  their  quandle product  $(X, *_1) \circ (X, *_2) = (X, *_1 *_2)$ is an algebraic system on $X$ with the operation
$$
x (*_1 *_2) y = (x *_1 y) *_2 y,~~x, y \in X.
$$
In general this algebraic system is not a quandle. But if additionally $*_2$ is distributive with respect to $*_1$ 
$$
(x *_1 y) *_2 z = (x *_2 z) *_1 (y *_2 z),\qquad  x, y, z \in X,
$$
then $(X, *_1 *_2)$ is a quandle. If $*_2$ is distributive with respect to $*_1$ and vice versa, then we can define an operation $*_1^{n_1} *_2^{m_1} \ldots *_1^{n_k} *_2^{m_k}$, where $n_i$ and $m_i$ are integers. This operation defines a quandle on $X$ and there is an abelian group structure on the set of these quandles.  The unit element of this group is the trivial quandle on $X$. This result gives an explanation of the axiom (M1) in  Definition \ref{many-rack}.
\end{remark}

\bigskip


\subsection{$2$-multi-racks on $\mathbb{Z}$}

In this subsection we construct some rack operations on the set of integers $\mathbb{Z}$.
The operation $x * y = 2y-x$ gives a quandle structure on  $\mathbb{Z}$ and $(\mathbb{Z}, *)$ is the core quandle $\Core \mathbb{Z}$. Let us consider a more general operation
$$
x \underset{\varepsilon, a, b}{*} y = \varepsilon x + a y + b,~~x, y \in \mathbb{Z}
$$
for some $\varepsilon \in \{ \pm 1 \}, a, b \in \mathbb{Z}.$  We can consider this operation as a deformation of the operation $*$.

%
%


\begin{proposition} \label{ZRT}
1) The set of integers $\mathbb{Z}$ with an operation
$$
x \underset{\varepsilon, a, b}{*} y = \varepsilon x + a y + b,~~x, y \in \mathbb{Z},~~\varepsilon \in \{ \pm 1 \}, a, b \in \mathbb{Z},
$$
is a quandle if and only if it is the trivial quandle: $\varepsilon = 1$, $b = 0$, or it  is the core quandle with the operation $x * y = - x + 2 y.$

2) On the set of integers the operation
$$
x \underset{\varepsilon,  b}{*} y  = \varepsilon x + b,~~x, y \in \mathbb{Z}
$$
defines a rack for $ \varepsilon \in \{ \pm 1 \}$, $b \in \mathbb{Z}$. 
\end{proposition}

We will call the rack operations from this proposition by linear operations.

The next proposition gives all 2-multi racks on $\mathbb{Z}$ with linear operations. 

\begin{proposition} \label{mr}
Let $b$, $b_1$, and $b_2$ are integer numbers. The next pairs of multiplications define 2-multi-rack on $\mathbb{Z}$:

1) $x *_1 y = x$,~~~ $x *_2 y =  x$;

2) $x *_1 y = 2 y - x$, ~~~ $x *_2 y =  x$;

3.1) $x *_1 y = x + b_1$, ~~~ $x *_2 y  = x + b_2$;

3.2) $x *_1 y = x$, ~~~ $x *_2 y = - x + b$;

3.3) $x *_1 y = -x + b$, ~~~ $x *_2 y = x$;

3.4) $x *_1 y = -x + b$,~~~ $x *_2 y =- x + b$.

\end{proposition}

\begin{proof}

As follows from Proposition \ref{ZRT}, all these operations are rack operations. We need to check the distributivity axioms,
\begin{equation} \label{mixdis}
(x *_1 y) *_2 z = (x *_2 z) *_1 (y *_2 z),~~(x *_2 y) *_1 z = (x *_1 z) *_2 (y *_1 z),~~x, y, z \in \mathbb{Z}.
\end{equation}

For the trivial multiplications from 1) these axioms are evidently true.

2) The first operation defines the Core quandle and the second operation defines the trivial quandle. Hence, the distributivity  axioms have the form,
$$
x *_1 y = x  *_1 y,~~x  *_1 z =x *_1 z.
$$
We have identities.

3)  By Proposition \ref{ZRT}, the following two operations define racks on $\mathbb{Z}$,
$$
x *_1 y = \varepsilon_1 x + b_1,~~~x *_2 y = \varepsilon_2 x + b_2,~~~x, y \in \mathbb{Z}.
$$
 Further we need to find conditions on $\varepsilon_1, \varepsilon_1 \in \{ \pm 1 \}$ and  $b_1, b_2 \in \mathbb{Z}$ under which the axioms~(\ref{mixdis}) fulfill.
It is easy to see that from these  axioms  follows
$$
 \varepsilon_2 b_1 + b_2 =  \varepsilon_1 b_2 + b_1.
$$
Depending on the values of  $\varepsilon_i$, consider 4 cases:

{\it Case 1}:  $\varepsilon_1 = \varepsilon_2 = 1$. In this case we get the equality 
$$
b_1 + b_2 = b_2 + b_1,
$$
which is true for any $b_1$ and $b_2$. We get the  2-multi-rack from 3.1.

{\it Case 2}:  $\varepsilon_1 = 1$, $\varepsilon_2 = -1$. In this case $b_1 = 0$ and we have the  2-multi-rack from 3.2.

{\it Case 3}:  $\varepsilon_1 = -1$, $\varepsilon_2 = 1$. This case is similar to the previous one. We have the case 3.3.

{\it Case 4}:  $\varepsilon_1 = \varepsilon_2 = -1$. In this case $b_1 = b_2$ and we have the  2-multi-rack from 3.4.

To prove that this is the only 2-multi-racks on $\mathbb{Z}$ with linear operations, one needs to consider all other pair of operations from Proposition \label{ZRT} and to check the axioms (\ref{mixdis}).
\end{proof}

\subsection{2-multi-racks and solutions for the Braid equation}

It is known that any rack $(X, *)$ defines a solution $R(x, y) = (x * y, x)$  to the Braid equation:
$$
(R \times \id) (\id \times R) (R \times \id) = (\id \times R) (R \times \id)  (\id \times R).
$$
We can interpret  $(X, *)$ is a 2-multi-rack $(X, *_1, *_2)$, where $ x *_1 y=x * y$, and the second multiplication is trivial, $x *_2 y = x$.

\begin{question}
For what 2-multi-racks $(X, *_1, *_2)$ with non-trivial operations the map $R \colon X \times X \to X \times X$, $R(x, y) = (x *_1 y, x *_2 y )$ gives a solution to the Braid equation?
\end{question}
The general conditions on these operations can be found in \cite{BCEIKL}. Analizing linear 2-multi-racks from Proposition \ref{mr}, we get

\begin{proposition} \label{BE}
The following maps $R \colon \mathbb{Z} \times \mathbb{Z} \to \mathbb{Z} \times \mathbb{Z}$ give solutions to the Braid equation on $\mathbb{Z}$:

1) $R(x, y) =( x,  y)$;

2.1) $R(x, y) = (2 y - x,  x)$;

2.2) $R(x, y) = (y,  2x - y)$;

3) $R(x, y) = (x,  - x + b)$,

where $b$, $b_1$, and $b_2$ are integer numbers.
\end{proposition}

\begin{remark}
This proposition holds if we replace ~$\mathbb{Z}$ by any abelian group.
\end{remark}


\bigskip


\section{Rack bialgebras and applications} \label{RBA}

\subsection{Bialgebras}
%

Let us recall the definition of the group algebra for an $n$-valued group (see \cite{BR1}). If $X=\{x_i\}$ is an $n$-valued group with the multiplication
$$
x * y = [z_1, z_2, \ldots, z_n],
$$
then the group algebra $\Bbbk[X]$ is a vector space over   $\Bbbk$ 
with a basis identified with elements of $X$ and the multiplication, which is defined on the basis vectors by the rule:
\bea
\label{mult}
x \cdot y = z_1 +  z_2 +  \ldots +  z_n ,~~x, y \in X.
\eea
We can consider the dual object $C(X)$ - the algebra of functions $f \colon X\rightarrow \Bbbk$, 
then it is equipped with the coassociative comultiplication 
\bea
\Delta \colon C(X)\rightarrow C(X)\otimes C(X)\nn
\eea
defined as follows
\bea
\label{comult}
\Delta f (x,y)=\sum_{i=1}^n f(z_i).
\eea
The natural nondegenerate pairing  $\Bbbk[X]\otimes C(X)\rightarrow \Bbbk$ is given by
\bea
\left( \sum_i \alpha_i x_i,   f\right) = \sum_i \alpha_i f(x_i),~~~\alpha_i \in \Bbbk.   \nn 
\eea
The comultiplication $\Delta$ is the dual map for the multiplication (\ref{mult}). The coassociativity of $\Delta$ and associativity of the multiplication (\ref{mult}) is the consequence of the associativity of the $n$-valued group $X$.

%
Let us recall a definition for an $n$-homomorphism from \cite{B}. For a pair $(A, B)$, where  $A$ is an algebra and $B$ is a commutative algebra one defines a so-called trace map $f \colon A\rightarrow B$ that is a linear homomorphism (that is a homomorphism of vector spaces) such that $f(ab)=f(ba)$. Then one could define a series of derived poly-linear homomorphisms of a trace map $f$:
\bea
\Phi_k(f) \colon \bigotimes^{k} A\rightarrow B\nn
\eea
by the following recurrent procedure
\bea
\Phi_1(f)&=&f, \nn\\ 
\Phi_2(f)(a_1,a_2)&=&f(a_1)f(a_2)-f(a_1 a_2),\nn\\
\ldots &&\ldots\nn\\
\Phi_{k+1}(f)(a_1,\ldots, a_{k+1})&=&f(a_1)\Phi_k(f)(a_2,\ldots,a_{k+1})-\sum_{i=2}^{k+1}\Phi_k(f)(a_2,\ldots,a_1 a_i,\ldots, a_{k+1}),\nn\\
\ldots &&\ldots\nn
\eea
\begin{definition}
A linear homomorphism $f \colon A\rightarrow B$ is called Frobenius $n$-homomorphism if the following two conditions hold
\bea
(1) & f(1) &=n;\nn\\
(2) & \Phi_{n+1}(f)&=0.\nn
\eea
\end{definition}

\begin{remark}
$C(X)$ is an algebra and a coalgebra, but the way how these structures are consistent is different from such for bialgebras (for the Hopf algebras in particular). In this case the comultiplication is an $n$-homomorphism of the multiplication structure  (Lemma 12 of \cite{B}).
\end{remark}

\subsection{Spaces of invariant functions via $n$-Hopf structures}
The notion of an $n$-Hopf bialgebra gets quite natural interpretation in terms of the spaces of invariant functions. 
Suppose that $G$ is a group and $B$ is its subgroup. We will denote by $C(G)^B$ the set of $\Bbbk$-valued functions  on $G$ which are invariant under the conjugations by elements of $B$, i.e.
$$
C(G)^B = \{ f \colon G \to \Bbbk~|~f(g) = f(b^{-1} g b )~ \mbox{for any} ~ g \in G, b \in B \}.
$$
The natural coproduct   $\Delta$ on $C(G)$  is defined by the rule
$$
\Delta f(x,y)=f(x\cdot y), ~~f \in C(G). 
$$ 
This could not in general be restricted to the space $C(G)^B.$  However it is known a
%
%
%

\begin{proposition}
Let $B$ be a central subgroup of $G$, then a coproduct $\Delta \colon C(G)^B \to  C(G)^B \otimes C(G)^B$ can be defined by the rule
$$
\Delta f(x,y)=f(x\cdot y), ~~f \in C(G)^B.
$$
\end{proposition}

In the case of a finite subgroup $B$ (there is a version for a compact subgroup) we have:
\begin{proposition} 
Suppose that $B$ is a finite subgroup of $G$. Then for any
 $f \in C(G)^B$ the  coproduct:
$$
\Delta f(x,y) = \sum_{b \in B} f(b^{-1} x b \cdot y)
$$
lies in $C(G)^B \otimes C(G)^B$.
\end{proposition}

\begin{proof}

Check that the function $\Delta f$ is invariant under conjugation by $h \in B$ with respect to the first argument:

$$\Delta f(h^{-1} x h, y) = \sum_{b \in B} f(b^{-1}h^{-1} x  h b \cdot y)=\sum_{b^{\prime } \in B} f((b^{\prime})^{-1} x b^{\prime} \cdot y) = \Delta f(x,y).$$

Check that the function  $\Delta f$ is invariant under conjugation with respect to the second argument:

$$\Delta f(x, h^{-1} y h)=\sum_{b \in B} f(b^{-1}x b  \cdot  h^{-1} y  h )= \sum_{ b \in B} f(h b^{-1}x b h^{-1}  \cdot  y )=\sum_{b^{\prime } \in B} f((b^{\prime})^{-1} x b^{\prime} \cdot y) = \Delta f(x,y).$$
Here, in the second equality, we used the fact that $f$ lies in $C(G)^B$.

\end{proof}
 
 \begin{remark}
 This construction rephrase the coset construction of an $n$-valued group and its group $n$-Hopf algebra.
 \end{remark}
%
%
%

\begin{example}
Let $G = SL_2(\mathbb{Z}_2)$ be the special linear group over the 2-elements field.  
Write down all elements of $G$  and their orders:

$$E=\left( \begin{matrix}1&0\\ 0&1\end{matrix} \right);$$
$$A_1=\left( \begin{matrix}0&1\\ 1&0\end{matrix} \right) , \, Ord(A_1)= 2;
\, \, \,\, \, \, A_2=\left( \begin{matrix}1&1\\ 0&1\end{matrix} \right)  , Ord(A_2)=2;$$
$$A_3=\left( \begin{matrix}1&0\\ 1&1\end{matrix} \right); \, Ord(A_3)=2;
$$
$$
C_1=\left( \begin{matrix}1&1\\ 1&0\end{matrix} \right), \, Ord(C_1)=3;~~
C_2=\left( \begin{matrix}0&1\\ 1&1\end{matrix} \right), \, Ord(C_2)=3.
$$

Then
$$G=SL_2(\mathbb{Z}_2)=GL_2(\mathbb{Z}_2)= \left\{ E, A_1, A_2, A_3, C_1, C_2 \right\}.$$
The Borel subgroup of $G$ is $B=\left\{ E, A_2  \right\}.$

Note further that 
$$A_1 A_3=C_1,~~A_1 A_2 = C_2,~~C_1^{-1}=C_2.
$$
It is easy to see that the map  $SL_2(\mathbb{Z}_2) \to S_3$ to the symmetric group $S_3$, which is defined  by the rules,
$$A_1 \longmapsto  (12), A_2 \longmapsto (23), A_3 \longmapsto (13)$$
$$C_1 \longmapsto  (123), C _2 \longmapsto (132),$$
is an isomorphism.

Consider the  action of group $B$ on $G$ by conjugation. We get four orbits,
$G / B =\left\{ \bar{e}, \bar{(12)}, \bar{(23)}, \bar{(123)}\right\}.$ Hence, $C(G)^B = \mathbb{Z}_2^4$.
This example coincides with our coset-type example \ref{coset-Q}.
\end{example}

\subsection{Rack bialgebras}

To each finite quandle $(Q, *)$ one can associate a bialgebra $\Bbbk[Q]$ which is a linear space those basis we denote by the same letters $e_i$  as elements of $Q$. It is equipped with a bilinear multiplication defined on generators as follows
\bea
m(e_i,e_j)=e_i\ast e_j;\nn
\eea
and a linear comultiplication
\bea
\Delta e_i=e_i\otimes e_i.\nn
\eea


It is easy to check that $\Delta$ and $m$ satisfy the conditions:
\begin{enumerate}
\item The comultiplication $\Delta$ is coassociative
\bea
(\Delta\otimes \id)\Delta=(\id\otimes \Delta)\Delta.\nn
\eea
\item
The multiplication is a homomorphism of coalgebras
\bea
\Delta m= (m\otimes m)(\id\otimes \sigma \otimes \id) (\Delta\otimes \Delta)
\eea
where $\sigma$ is a permutation of the corresponding tensor components;
\item
The self-distributivity condition of $\ast$ could be expressed as the compatibility condition for the multiplication and comultiplication
\bea
\label{self-dist1}
m (m\otimes \id) =m (m\otimes m) (\id\otimes \sigma \otimes \id)(\id\otimes\ id \otimes \Delta).
\eea
\end{enumerate}

The self-distributivity condition (\ref{self-dist1}) is equivalent to the formula
$$
(c * b) * a = \sum_{(a)} (c * a^{(1)}) * (b * a^{(2)})
$$
where we used the Sweedler notations:
\bea
\Delta a=\sum_i a_i^{(1)}\otimes a_i^{(2)}=\sum a^{(1)}\otimes a^{(2)}.
\eea

\begin{remark}
A similar approaches to rack bialgebras can be found in \cite{ABRW} and \cite{CCES}.
\end{remark}

In analogy with the group case let us define the dual structure on the space of functions on a rack $Q$ with values in $\Bbbk$. We denote this space as $C(Q).$ The roles of operations change, $C(Q)$ has a commutative and associative pointwise multiplication
\bea
(fg)(x)=f(x)g(x),\qquad f,g\in C(Q), x\in Q.\nn
\eea
We denote it also as $m \colon C(Q)\otimes C(Q)\rightarrow C(Q).$
The comultiplication is dual to the quandle multiplication
\bea
\Delta f (x,y)=f(x\ast y).\nn
\eea
These operations are compatible in the following sense:
\begin{enumerate}
\item
The comultiplication is a homomorphism of algebras
\bea
\label{comp}
 (m\otimes m)(id\otimes \sigma \otimes id) (\Delta\otimes \Delta)=\Delta m,
\eea
where $\sigma$ as before is a permutation of the corresponding tensor components;
\item

The self-dictributivity condition of $\ast$  in $Q$ is equivalent to the following condition on operations in $C(Q)$
\bea \label{self-dist}
(\Delta\otimes id)\Delta=(id\otimes id\otimes m)(id\otimes \sigma \otimes id)(\Delta\otimes \Delta)\Delta.
\eea
\end{enumerate}
\begin{definition}
\label{corack}
We call a linear space $A$ a corack bialgebra if $A$ is equipped with two operations $m \colon A\otimes A\rightarrow A$ and $\Delta \colon A\rightarrow A\otimes A$ such that $m$ is associative and the conditions (\ref{comp}) and (\ref{self-dist}) are satisfied.
\end{definition}
We can generalize the definition for the $n$-bialgebras to the case of corack bialgebras
\begin{definition}
We call a linear space $A$ an $n$-corack bialgebra if in the definition \ref{corack} we change the condition (\ref{comp}) by requiring that $\Delta$ is an $n$-homomorphism of $m$.
\end{definition}

\begin{theorem}
Let $Q$ be an $n$-valued rack, then the algebra of functions $C(Q)$ is an $n$-corack bialgebra.
\end{theorem}
\begin{proof}
The only thing to prove is the condition that $\Delta$ is an $n$-homomorphism of $C(Q)$ as an algebra. The demonstration goes through the same lines as in the case of $n$-groups, but we provide it here for the sake of completeness. Let us calculate the derived comultiplications, the 2-nd one is: 
\bea
\Delta^{(2)}(f_1,f_2)=\Delta(f_1)\Delta(f_2)-\Delta(f_1 f_2).\nn
\eea
This is an element in $C(Q)\otimes C(Q).$ We associate this with the space of functions on two arguments. Let us recall that
\bea
x\ast y=[ z_1, z_2, \ldots, z_n ]\nn
\eea
and
\bea
\Delta(f)(x, y)=\sum_i f(z_i)=\sum_i \lambda_i,~~~\lambda_i \in \Bbbk.
\eea
Let us denote
\bea
f_i(z_j)=\lambda_j^i, ~~~\lambda_j^i \in \Bbbk, ~~i = 1, 2, \ldots,~~j = 1, 2, \ldots, n.\nn
\eea
Then
\bea
\Delta^{(2)}(f_1,f_2)(x, y)&=&f_1(x \ast y) f_2(x \ast y)-(f_1 f_2)(x \ast y)\nn \\ \\
&=&(\sum_i \lambda^1_i)(\sum_j \lambda^2_j)-\sum_i \lambda^1_i \lambda^2_i=\sum_{i\neq j}\lambda_i^1\lambda_j^2.\nn
\eea
This is exactly the polarized second elementary symmetric polynomial.
Let us demonstrate that for general $k$ the derived comultiplication takes the form
\bea\label{derk}
\Delta^{(k)}(f_1,\ldots,f_k)(x, y)=\sum_{j_1,\ldots,j_k: j_l\neq j_m} \lambda^1_{j_1}\ldots \lambda^k_{j_k}.
\eea
This implies in particular that $\Delta^{(n+1)}=0.$ We demonstrate the formula (\ref{derk}) by induction. Let us express $\Delta^{(k+1)}$ by definition:
\bea
\Delta^{(k+1)}(f_1,\ldots,f_{k+1})(x, y)&=&\Delta(f_1)(x, y)\Delta^{(k)}(f_2,\ldots,f_{k+1})(x, y)\nn\\
&-&\sum_{i=1}^k\Delta^{(k)}(f_2,\ldots,f_i,f_1 f_{i+1},f_{i+2},\ldots,f_{k+1})\nn\\
&=&\sum_{j_1}\lambda^1_{j_1}\sum_{j_2,\ldots,j_{k+1}: j_l\neq j_m} \lambda^2_{j_2}\ldots \lambda^{k+1}_{j_{k+1}}\nn\\
&-&\sum_{i=1}^k \sum_{j_2,\ldots,j_{k+1}: j_l\neq j_m} \lambda^1_{j_{i+1}}\lambda^2_{j_2}\ldots \lambda^{k+1}_{j_{k+1}}.\nn
\eea
The subtraction part of this formula is exactly subtraction of terms of the first summand such that the index $j_1$ coincide with one of the indices $\{j_2,\ldots,j_{k+1}\}.$
\end{proof}


\bigskip

\section{Conclusion}
\begin{itemize}

\item

We expect to develop topological applications of the presented constructions in the near future. This mainly concerns the multivalued quandles of a pair of topological spaces proposed in Turaev's work \cite{T}. We also plan to use the connection of multivalued quandles with various theories of virtual knots and virtual braid groups, including virtual knots with several types of virtual intersections proposed by Kauffman \cite{Kauffman}

\item
The second direction of development of the described tools is the study of braided (noncommutative) version of the Buchstaber-Rees theorem for a set with the action of a braid group. Such objects are widely studied in the theory of quantum groups and its generalizations, for example in the works \cite{Gurevich}.

\item
We place special hopes on the development of the representation theory of racks and quandles and we expect that the corresponding bialgebras: rack, corack and their $n$-valued versions, plays a substantial role in this theory in parallel to the role of the group algebra in Burnside lemma. This could open up new perspectives in the problem of constructing topological invariants of knots and its generalizations.

\end{itemize}

%
%
%
%
%
%

\subsection*{Conflicts of interest}
The authors declare no conflicts of interest.

\begin{ack}
The parts 1, 5 of the work was carried out with the support of the Russian Science Foundation grant 20-71-10110. 
The parts 2, 3 of the work is supported  by the Ministry of Science and Higher Education of Russia (agreement No. 075-02-2023-943).
The section $4$ of this work is supported  by the Theoretical Physics and Mathematics Advancement Foundation BASIS No 23-7-2-14-1. 

We thank V. M. Buchstaber for useful discussions.
\end{ack}

\end{document}